\documentclass
{amsart}

\usepackage{graphicx}
\usepackage{pict2e}
\usepackage{amssymb}
\usepackage{amsthm}                
\usepackage[margin=1in]{geometry}  
\usepackage{epstopdf}
\usepackage{amsmath}
\usepackage{subfigure}
\usepackage{latexsym}
\usepackage{amsfonts}
\usepackage{amssymb}
\usepackage{tikz}
\usepackage{mathdots}
\usepackage{comment}
\usepackage{float}
\usepackage[mathscr]{euscript}
\usepackage{enumerate}
\usepackage{epsfig}
\usepackage { hyperref }
\usepackage { graphicx }
\usepackage{xparse}
\usepackage[utf8]{inputenc}
\usepackage{longtable}

\usepackage[lite]{amsrefs}

\theoremstyle{definition}
\newtheorem{theorem}{Theorem}[section]
\newtheorem{lemma}[theorem]{Lemma}

\newtheorem{proposition}[theorem]{Proposition}

\theoremstyle{definition}
\newtheorem{definition}[theorem]{Definition}
\newtheorem{example}[theorem]{Example}
\theoremstyle{remark}

\numberwithin{equation}{section}

\numberwithin{equation}{section}

\begin{document}

\title{ Singular Knots and Involutive Quandles}

\author{Indu R. U. Churchill}
\address{Department of Mathematics, 
	University of South Florida, Tampa, FL 33620, U.S.A.}
\email{udyanganiesi@mail.usf.edu}

\author{M. Elhamdadi}
\address{Department of Mathematics, University of South Florida, 
Tampa, FL 33647 USA}
\email{emohamed@math.usf.edu }

\author{M. Hajij}
\address{Department of Mathematics, University of South Florida, 
Tampa, FL 33647 USA}
\email{mhajij@usf.edu}

\author{Sam Nelson}
\address{Department of Mathematics, 
	Claremont McKenna College,
	850 Columbia Ave.
	Claremont, CA 91711, USA}
\email{Sam.Nelson@claremontmckenna.edu}





\begin{abstract}
The aim of this paper is to define certain algebraic structures coming from generalized Reidemeister moves of singular knot theory.  We give examples, show that the set of colorings by these algebraic structures is an invariant of singular links.  As an application we distinguish several singular knots and links.  
\end{abstract}

 \maketitle

 \tableofcontents
\section{Introduction}

The study of singular knots and their invariants was motivated mainly by the theory of Vassiliev invariants \cite{Vassiliev}.  Most of the important knot invariants have been extended to singular knot invariants. Fiedler extended the Kauffman state models of the Jones and Alexander polynomials to the context of singular knots \cite{Fiedler}. Juyumaya and Lambropoulou constructed a Jones-type invariant for singular links using a Markov trace on a variation of the Hecke algebra \cite{JL}. In \cite{KV} Kauffman and Vogel defined a polynomial invariant of embedded $4$-valent graphs in $\mathbb{R}^3$ extending an invariant for links in $\mathbb{R}^3$ called  the Kauffman polynomial \cite{kaufgraph}. In \cite{HN}, Henrich and the fourth author investigated singular knots in the context of virtual knot theory, flat virtual knot theory and flat singular virtual knot theory.  They introduced algebraic structures called semiquandles, singular semiquandles and virtual singular semiquandles.  They also gave an application to distinguishing Vassiliev-type invariants of virtual knots.  Other extensions of classical invariants of knots to singular knots can be found in the works of \cite{ St, Paris}. 
In this article we consider colorings of singular knots and links by certain algebraic structures.  As in the case of classical knot theory \cite{Prz} we show that the set of colorings is independent of the choice of the diagrams of a given singular knot or link making it an invariant of singular knots.  We show that this invariant is computable.  We then use it to distinguish many singular knots.

This article is organized as follows.  In section~\ref{sec2}, we review the basics of quandles and give examples.  Section~\ref{sec3} gives the definition of {\it singquandles} with examples and shows that the set of colorings of singular links by a singquandle is an invariant of singular links.  In section~\ref{sec4} we compute the coloring invariants for many singular links and use it to distinguish many singular links.  In section~\ref{sec5} we collect some open questions for future research.

\section{Acknowledgements}

The fourth listed author was partially supported by Simons Foundation Collaboration Grant 316709.

\section{Review of Quandles}\label{sec2}
We review the basics of quandle theory needed for this article.    Quandles are non-associative algebraic structures that correspond to the axiomatization of the three Reidemeister moves in knot theory. Since the early eighties when quandles were introduced by Joyce~\cite{Joyce} and Matveev~\cite{Matveev} independently, there has been a lot of interest in the theory, (see for example the book \cite{ EN} and the references therein). Joyce and Matveev proved that the {\it fundamental} quandle of a knot is a complete invariant up to orientation.  Precisely, given two knots $K_0$ and $K_1$, the fundamental quandle $Q(K_0)$ is isomorphic to the fundamental quandle $Q(K_1)$  if and only if $K_1$ is equivalent to $K_0$ or $K_1$ is equivalent to the reverse of the mirror image of $K_0$.  Quandles have been used by topologists to construct invariants of knots in $3$-space and knotted surfaces in $4$-space.  For example in \cite{HN}, Henrich and the fourth author investigated singular knots in the context of virtual knot theory. Their derived algebraic structure is called {\it virtual semiquandle} and it comes from some generalizations of Reidemeister moves for virtual knots.  They used it to construct invariants to distinguish generalized knots, with an application to distinguishing Vassiliev-type invariants of virtual knots.  Now we recall the basic definitions of quandles and give a few examples.

\begin{definition}
	\textup{A \textit{rack} is a set $X$ with a binary operation $*$ satisfying the following two axioms:}%
	\begin{list}{}{}
		\item[\textup{(i)}]{for all $x\in X$, the right multiplication  $ y \mapsto y * x$  by $x$ is a bijection,} \textup{and}
		\item[\textup{(ii)}]{$(x* y)* z=(x* z)*(y* z)$}.
	\end{list}
	\textup{A rack which further satisfies $x* x=x$ for all $x\in X$ is called a \textit{quandle}.}
\end{definition}

 A {\it quandle homomorphism} between two quandles $X, Y$ is
 a map $f: X \rightarrow Y$ such that $f(x* y)=f(x) * f(y) $.
 A {\it quandle isomorphism} is a bijective quandle homomorphism, and 
 two quandles are {\it isomorphic} if there is a quandle isomorphism 
 between them.  The right multiplication $R_x$ is the automorphism of $X$ given by $R_x(y)=y*x$.  Its inverse will be denoted by $R_x^{-1}(y):=y*^{-1}x$.  The set of all quandle isomorphisms of $X$ is a group denoted $Aut(X)$.  Its subgroup generated by all right multiplications $R_x$ is called the Inner group and denoted $Inn(X)$.  A quandle is called {\it involutive} if $(x*y)*y=x$ for all $x, y \in X$.  In other words $R_x=R_x^{-1}, \; \forall x \in X$.  
 
 Typical examples of quandles include the following. 
 \begin{itemize}
 	\setlength{\itemsep}{-.9pt}
 	\item
 	Any non-empty set $X$ with the operation $x* y=x$ for all $x,y \in X$ is
 	a quandle called the {\it trivial} quandle.
 	
 	\item
 	A group $X=G$ with
 	$n$-fold conjugation
 	as the quandle operation: $a * b=b^{-n} a b^n$.
 	
 	\item
 	A group $X=G$ with the binary operation: $a * b=ba^{-1} b$ is a quandle.  It is called the {\it core} quandle of the group $G$.
 	
 	\item
 	Let $n$ be a positive integer.
 	For  
 	$a, b \in \mathbb{Z}_n$ (integers modulo $n$), 
 	define
 	$a  *  b \equiv 2b-a \pmod{n}$.
 	Then $* $ defines a quandle
 	structure  called the {\it dihedral quandle},
 	$R_n$.
 	This set can be identified with  the
 	set of reflections of a regular $n$-gon
 	with conjugation
 	as the quandle operation.
 	\item
 	Any ${ \mathbb{Z} }[T, T^{-1}]$-module $M$
 	is a quandle with
 	$a * b=Ta+(1-T)b$, $a,b \in M$, called an {\it  Alexander  quandle}.
 \end{itemize}
 In standard knot theory, to distinguish knots using quandles, one only needs to consider {\it connected} quandles according to section 5.2 in \cite{Ohtsuki} (see also lemma 3.1 in \cite{CESY}).  A quandle is {\it connected} if for every $x, y \in X$, there exists a sequence of elements 
 $x_1, \ldots, x_n \in X$ for some positive integer $n$ and $\epsilon_1, \epsilon_2, \cdots \epsilon_n \in \{-1,1 \}$  such that 
 $( \cdots ( x *^{\epsilon_1} x_1 )*^{\epsilon_2} x_2 ) *^{\epsilon_3} \cdots ) \cdots *^{\epsilon_n} x_n ) =y$.  In other words the Inner group $Inn(X)$ acts transitively on the quandle $X$.

\begin{figure}[H]
  \centering
   {\includegraphics[scale=0.1]{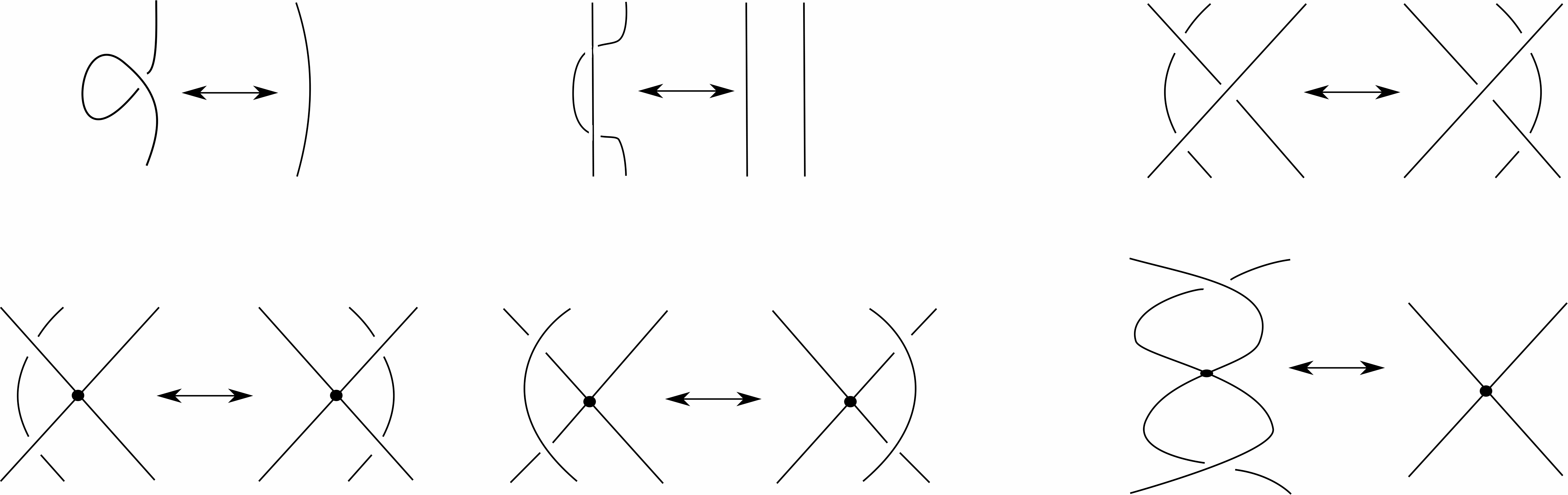}
   	\vspace{.2in}
     \caption{Classical Reidemeister moves $RI$, $RII$ and $RIII$ on the top and singular Reidemeister $RIVa$, $RIVb$ and $RV$ on the bottom.}
  \label{Rmoves}}
\end{figure}

\section{Singular Quandles}\label{sec3}
In this section we define the notion of {\it singquandles}, give some examples and use them to construct an invariant of singular knots and links.  The invariant is the set of colorings of a given singular knot or link by a singquandle.  The colorings of regular and singular crossings are given by the following figure~\ref{Rmoves}.

\begin{figure}[H]
	\tiny{
		\centering
		{\includegraphics[scale=0.25]{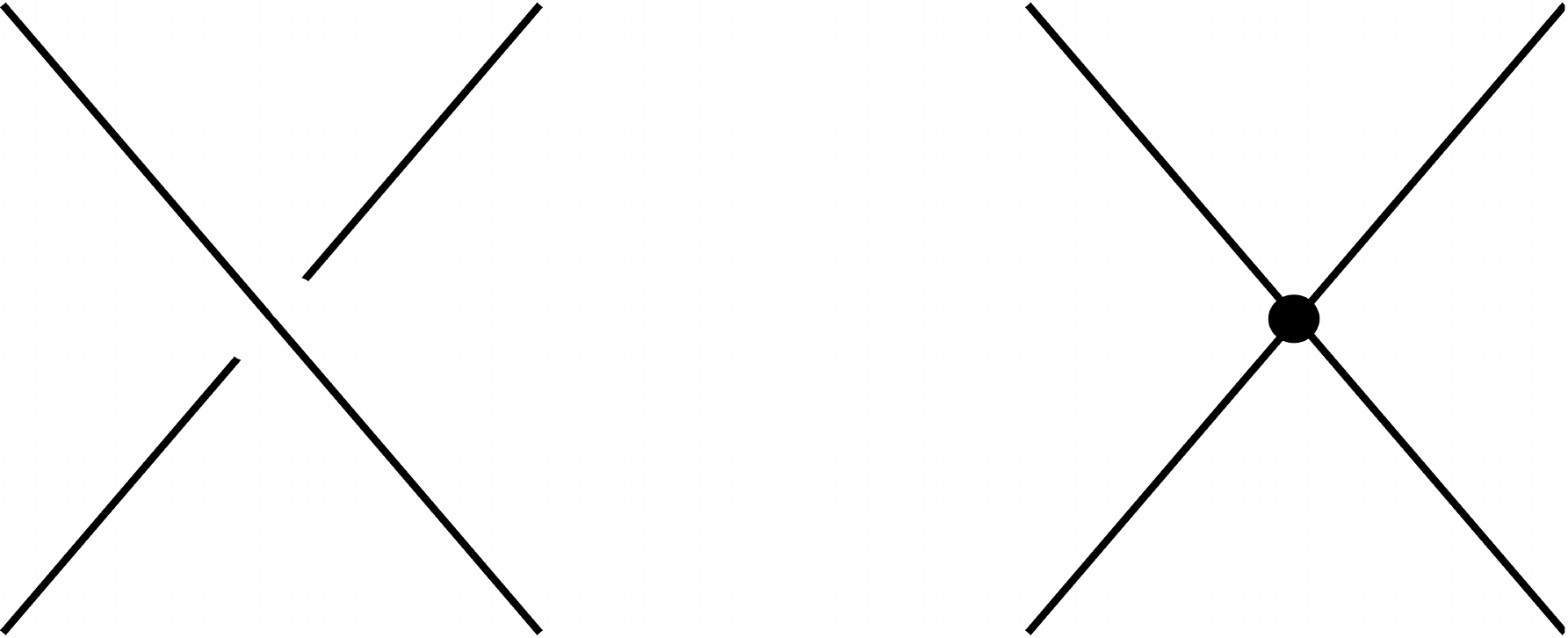}
			\put(-90,55){$y$}
			\put(-133,55){$x$}
			\put(-143,-10){$y*x$}
			\put(-90,-10){$x$}
			\put(-45,55){$x$}
			\put(-5,55){$y$}
			\put(-50,-10){$R_1(x,y)$}
			\put(-10,-10){$R_2(x,y)$}}
		\vspace{.2in}
		\caption{Regular and singular crossings}
		\label{Rmoves}}
\end{figure}

\begin{figure}[H]

  \centering
   {\includegraphics[scale=0.25]{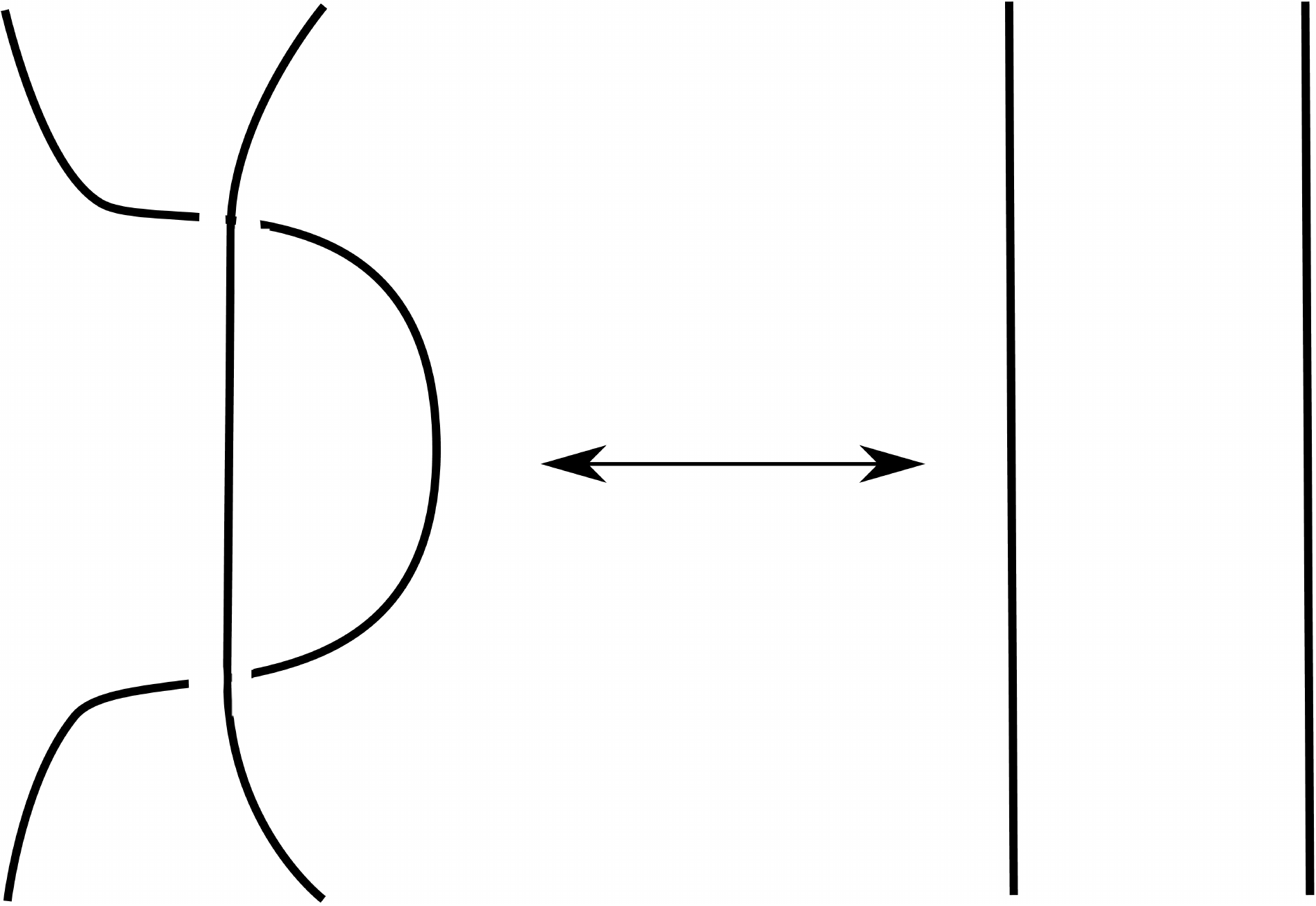}
	\put(-113,100){$y$}
    \put(-138,100){$x$}    
	\put(-3,100){$y$}
	\put(-3,-8){$y$}
	\put(-99,67){$x*y$}
	 \put(-178,-8){$(x*y)*y$}  
    \put(-33,100){$x$}    
    \put(-33,-8){$x$}
    \put(-125,50){$y$}
    \put(-108,-8){$y$}
    	\vspace{.2in}
     \caption{Reidemeister move II for regular crossings}}
\end{figure}

\begin{figure}[H]
\tiny{
  \centering
   {\includegraphics[scale=0.55]{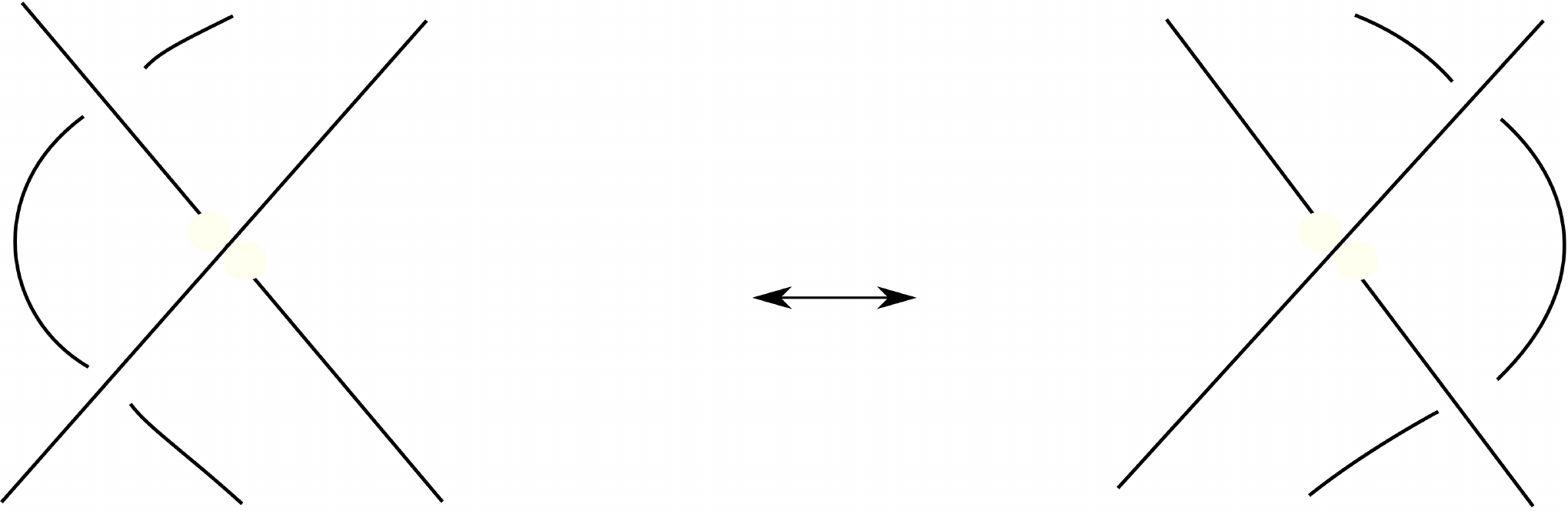}
    \put(-315,108){$x$}
    \put(-273,108){$y$}
    \put(-288,70){$x$}
    \put(-344,60){$y*x$}
    \put(5,60){$y*z$}
    \put(-339,-7){$z$}
    \put(-307,-7){$(y*x)*z$}
    \put(-235,-7){$x*z$}
    \put(-105,-7){$z$}
    \put(-85,-7){$(y*z)*(x*z)$}
    \put(-5,-7){$x*z$}
    \put(-238,108){$z$}
     \put(-83,108){$x$}
    \put(-38,108){$y$}
    \put(-5,108){$z$}
    }
    \vspace{.2in}
    \caption{Reidemeister move III for regular crossings}
    \label{}}
     
\end{figure}

\begin{figure}[H]
	\tiny{
		\centering
		{\includegraphics[scale=0.55]{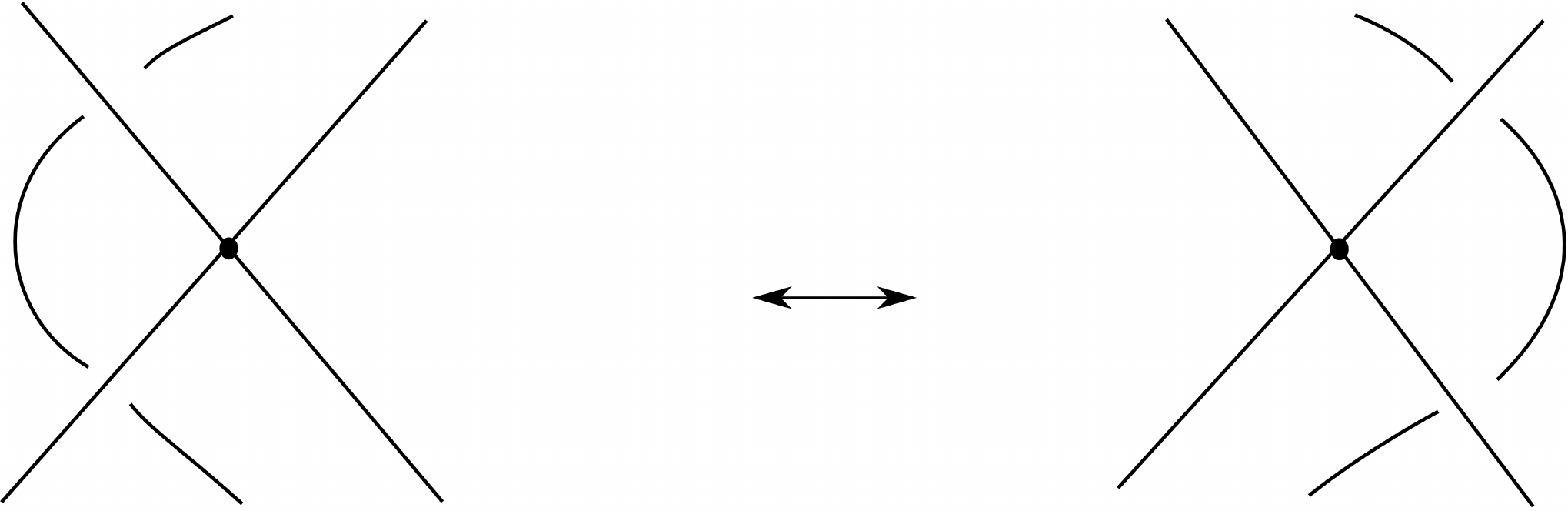}
			\put(-315,108){$x$}
			\put(-273,108){$y$}
			\put(-288,70){$x$}
			\put(-344,60){$y*x$}
			\put(5,60){$y*z$}
			\put(-359,-7){$R_1(x,z)$}
			\put(-307,-7){$(y*x)*R_1(x,z)$}
			\put(-235,-7){$R_2(x,z)$}
			\put(-125,-7){$R_1(x,z)$}
			\put(-85,-7){$(y*z)*R_2(x,z)$}
			\put(-5,-7){$R_2(x,z)$}
			\put(-238,108){$z$}
			\put(-83,108){$x$}
			\put(-38,108){$y$}
			\put(-5,108){$z$}
		}
		\vspace{.2in}
		\caption{The singular Reidemeister move RIVa and colorings }
		\label{The generalized Reidemeister move RIV}}
\end{figure}

\begin{figure}[H]
	\tiny{
		\centering
		{\includegraphics[scale=0.55]{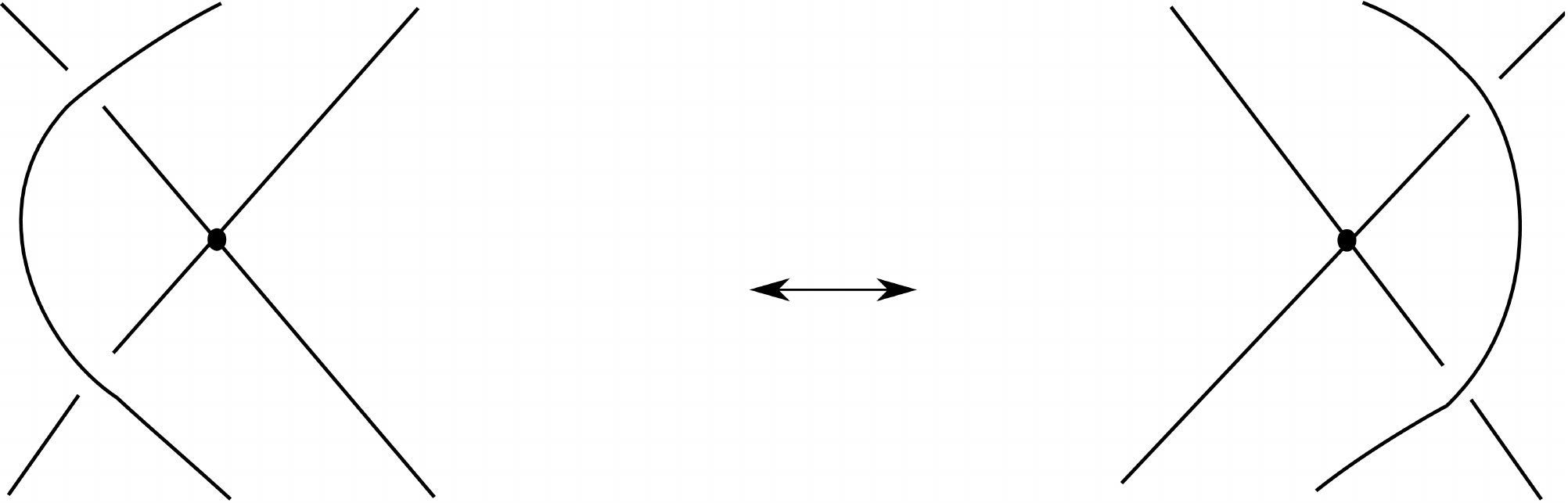}
			\put(-315,108){$x$}
			\put(-273,108){$y$}
			\put(-288,70){$x*y$}
			\put(-344,60){$y$}
			\put(5,60){$y$}
			\put(-359,-7){$R_1(x*y,z)*y$}
			\put(-270,-7){$y$}
			\put(-235,-7){$R_2(x*y,z)$}
			\put(-125,-7){$R_1(x,z*y)$}
			\put(-55,-7){$y$}
			\put(-5,-7){$R_2(x,z*y)*y$}
			\put(-238,108){$z$}
			\put(-83,108){$x$}
			\put(-38,108){$y$}
			\put(-5,108){$z$}
		}
		\vspace{.2in}
		\caption{The singular Reidemeister move RIVb and colorings }
		\label{The generalized Reidemeister move RIV}}
\end{figure}

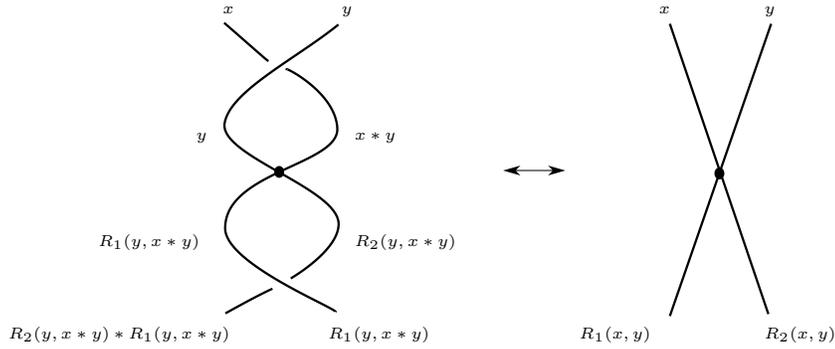
\begin{figure}[H]
\tiny{
  \centering
   {\includegraphics[scale=0.35]{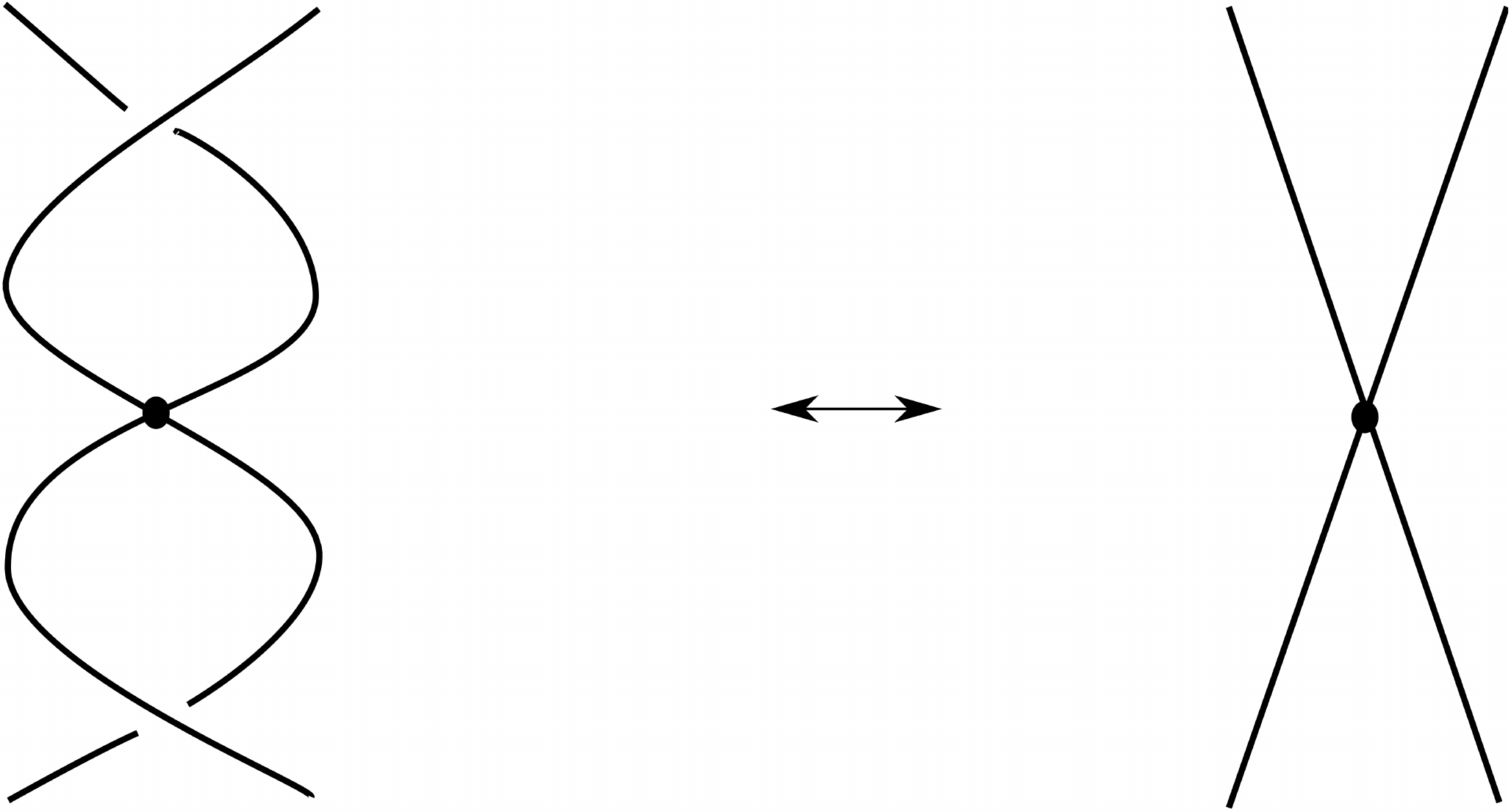}
	\put(-163,115){$y$}
    \put(-208,115){$x$}    
	\put(-3,115){$y$}
	\put(-3,-8){$R_2(x,y)$}
	\put(-218,67){$y$}
	\put(-158,67){$x*y$}
	\put(-255,27){$R_1(y,x*y)$}
	\put(-158,27){$R_2(y,x*y)$}
	 \put(-289,-8){$R_2(y,x*y)*R_1(y,x*y)$}  
    \put(-43,115){$x$}    
    \put(-73,-8){$R_1(x,y)$}
    \put(-168,-8){$R_1(y,x*y)$}}
     \vspace{.2in}
     \caption{The singular Reidemeister move RV and colorings}
     \label{The generalized Reidemeister move RV}}
\end{figure}

Since our singular crossings are unoriented, we need the operations to be symmetric in the sense that if we rotate the crossing in the right diagram of figure~\ref{Rmoves} by $90$, $180$ or $270$ degrees, the operations should stay the same in order for colorings to be well-defined.  Therefore we get the following three axioms:
\begin{eqnarray}
x &=& R_1(y,R_2(x,y)) = R_2(R_2(x,y),R_1(x,y)),\label{rotation1}\\
y &=& R_2(R_1(x,y),x) = R_1(R_2(x,y), R_1(x,y)), \label{rotation2}\\
(R1(x,y),R2(x,y))&=&(R_2(y,R_2(x,y)), R_1(R_1(x,y),x)).\label{rotation3}
\end{eqnarray}

The axioms of the following definition come from the generalized Reidemeister moves RIV and RV in the respective figures~\ref{The generalized Reidemeister move RIV} and~\ref{The generalized Reidemeister move RV}.

\begin{definition}\label{SingInvQdle}
	Let $(X, *)$ be an involutive quandle.  Let $R_1$ and $R_2$ be two maps from $X \times X$ to $X$.  The triple $(X, *, R_1, R_2)$ is called a {\it singquandle} if, in addition to the three axioms~\ref{rotation1}, ~\ref{rotation2} and~\ref{rotation3}, the following axioms are satisfied
	\begin{eqnarray}
		(y * z)* R_2(x, z) &=& (y*x)*R_1(x,z) \;\;\;\text{coming from RIVa} \label{eq1}\\
		R_1(x, y)                     & =&  R_2(y*x, x) \;\;\;\text{coming from RV}\label{eq2}\\
	      R_2(x,y)   &=& 	R_1(y*x, x)*R_2(y*x, x) \;\;\;\text{coming from RV} \label{eq3}\\
R_1(x*y,z)*y&=&R_1(x,z*y)  \;\;\;\text{coming from RIVb} \label{eq4}\\
R_2(x*y,z)&=&R_2(x,z*y)*y  \;\;\;\text{coming from RIVb} \label{eq5}	
\end{eqnarray}	
\end{definition}

As in the case of classical knot theory \cite{Prz}, the following straightforward lemma makes the set of colorings of a singular knot by singquandles an invariant of singular knots.

\begin{lemma}
	The set of colorings of a singular knot by a singquandle does not change by the moves RI, RII, RIII, $RIVa$, $RIVb$ and RV.
\end{lemma}



We end this section with a class of singquandles generalizing the class of involutive Alexander quandles.

\begin{proposition}
Let $\Lambda=\mathbb{Z}[t,B]/(t^2-1,B(1+t),t-(1-B)^2)$  and let $X$ be a 
$\Lambda$-module. Then the operations 
\[x\ast y=tx+(1-t)y, \quad R_1(x,y)=(1-t-B)x+(t+B)y 
\quad \mathrm{and}\quad R_2(x,y)=(1-B)x + By\]
make $X$ an involutive singquandle we call an \textit{Alexander singquandle}.
\end{proposition}

\begin{proof}
It is straightforward to  verify that $\ast$ makes $X$ an involutive quandle.
We note that since $t^2=1$, we have
\[t(1-t)=t-t^2=t-1\quad\mathrm{and}\quad (1-t)^2=1-2t+t^2=2(1-t),\]
since $B(1+t)=0$, we have $Bt=-B$ and since
$t-(1-B)^2=0$ we have
\[1-t-2B+B^2=0\quad\mathrm{and}\quad t+2B-B^2=1.\]
We verify that our operations satisfy the remaining singquandle axioms. 
First, we compute
\begin{eqnarray*}
R_1(y,R_2(x,y)) 
& = & (1-t-B)y+(t+B)((1-B)x+By) \\
& = & (1-t-B)y+(t+B)(1-B)x+(t+B)By \\
& = & (1-t-B+tB+B^2)y+(t+B)(1-B)x \\
& = & (1-t-B+tB+B^2)y+(t+B-tB-B^2)x \\
& = & (1-t-2B+B^2)y+(t+2B-B^2)x \\
& = & 0y+1x=x
\end{eqnarray*}
and
\begin{eqnarray*}
R_2(R_2(x,y),R_1(x,y)) 
& = &  (1-B)((1-B)x+By)+B((1-t-B)x+(t+B)y)\\
& = &  ((1-B)^2+B(1-t-B))x+B(1-B+t+B)y\\
& = &  (t+B-tB-B^2)x+B(1+t)y\\
& = &  (t+2B-B^2)x+B(1+t)y\\
& = &  1x+0y=x\\
\end{eqnarray*}
and (4.1) is satisfied.

Next, we have
\begin{eqnarray*}
R_2(R_1(x,y),x) 
& = & (1-B)((1-t-B)x+(t+B)y)+Bx \\
& = & ((1-B)(1-t-B)+B)x+(1-B)(t+B)y \\
& = & (1-t-B-B+tB+B^2+B)x+(t+B-tB-B^2)y \\
& = & (1-t-2B+B^2)x+(t+2b-B^2)y \\
& = & 0x+1y=y \\
\end{eqnarray*}
and
\begin{eqnarray*}
R_1(R_2(x,y),R_1(x,y)) 
& = & (1-t-B)((1-B)x+By)+(t+B)((1-t-B)x+(t+B)y)\\
& = & (1-t-B)(1-B+t+B)x+ (B(1-t-B)+(t+B)^2)y\\
& = & (1-t-B)(1+t)x+ (B-tB-B^2+t^2+2tB+B^2)y\\
& = & (1-t-B+t-t^2-tB)x+ 1y\\
& = & 0x+ 1y\\
\end{eqnarray*}
and axiom (4.2) is satisfied.

Continuing, we have
\begin{eqnarray*}
(R_2(y,R_2(x,y)),R_1(R_1(x,y),x)) 
& = & ((1-B)y+B((1-B)x+By), \\
& & \quad\quad(1-t-B)((1-t-B)x+(t+B)y)+(t+B)x)\\
& = & ((1-B+B^2)y+B(1-B)x,\\
& & \quad \quad((1-t-B)^2+t+B)x+(1-t-B)(t+B)y)\\
& = & ((t+B)y+(B-B^2)x,\\
& & \quad\quad(1-t-B-t+t^2+tB-B+tB+B^2+t+B)x\\
& & \quad \quad +(t-t^2-tB+B-tB-B^2)y)\\
& = & ((1-t-B)x+(t+B)y,(1-B)x+By)\\
& = & R(x,y)
\end{eqnarray*}
as required by axiom (4.3).

Next, we have
\begin{eqnarray*}
(y\ast z)\ast R_2(x,z) 
& = & t(ty+(1-t)z)+(1-t)((1-B)x+Bz) \\
& = & (1-t)(1-B)x+t^2y+(1-t)(t+B)z \\
& = & (1-t)(t+1-t-B)x+t^2y+(1-t)(t+B)z\\
& = & t(ty+(1-t)x)+(1-t)((1-t-B)x+(t+B)z) \\
& = & (y\ast x) \ast R_1(x,z)
\end{eqnarray*}
and we have (4.4). Next, we have
\begin{eqnarray*}
R_1(x,y) 
& = & (1-t-B)x+(t+B)y\\
& = & (1-t-B+tB+B)x+(t-tB)y\\
& = & ((1-B)(1-t)+B)x+(1-B)ty\\
& = & (1-B)(ty+(1-t)x)+Bx \\
& = & R_2(y\ast x, x)
\end{eqnarray*}
and we have (4.5). Continuing, we have
\begin{eqnarray*}
R_2(x,y) 
& = & (1-B)x+By\\
& = & (1-B+(1-t-B)(t-1)+(1-t)(1-t-B))x+(1-t-B+t+B-1+B)y\\
& = & ((1-t-B)(t-1)+1-B+(1-t)(1-t-B))x+(1-t-B+t-tB-1+B)y\\
& = & ((1-t-B)(t-1)+t^2+tB+(1-t)(1-t-B+tB+B))x\\  
& & \quad \quad +(1-t-B+(t-1)(1-B))y\\
& = & (t(1-t-B)(1-t)+t(t+B)+(1-t)((1-t)(1-B)+B))x+(t^2(1-t-B)\\
& & \quad \quad +t(1-t)(1-B))y\\
& = & t((1-t-B)(ty+(1-t)x)+(t+B)x)+(1-t)((1-B)(ty+(1-t)x)+Bx) \\
& = & R_1(y\ast x,x)\ast R_2(y\ast x,x)
\end{eqnarray*}
and we have (4.6).

Next, we  have
\begin{eqnarray*}
R_1(x\ast y,z) * y 
& = & t((1-t-B)(tx+(1-t)y)+(t+B)z)+(1-t)y \\
& = & t^2(1-t-B)x+ t(t+B)z+(1-t)(t(1-t-B)+1)y\\
& = & t^2(1-t-B)x+ t(t+B)z+(1-t)(t-t^2-tB+1)y\\
& = & t^2(1-t-B)x+ t(t+B)z+(1-t)(t+B)y\\
& = & (1-t-B)x+(t+B)(tz+(1-t)y) \\
& = & R_1(x,z\ast y)
\end{eqnarray*}
as required by (4.7).

Lastly, we have
\begin{eqnarray*}
R_2(x\ast y,z)
& = & (1-B)(tx+(1-t)y)+Bz\\
& = & t(1-B)x+t^2Bz+(1-t)(1+tB)y \\
& = & t((1-B)x+B(tz+(1-t)y))+(1-t)y \\
& = & R_2(x,z\ast y)\ast y
\end{eqnarray*}
as required.
Hence, all axioms in Definition \ref{SingInvQdle} are also satisfied.  This completes the proof.
\end{proof}

\begin{example}
Any abelian group $A$ becomes an involutive Alexander singquandle by choosing 
an involutive automorphism $t:A\to A$ and another homomorphism $B:A\to A$
satisfying the conditions $B(x)=-Bt(x)$ and $t(x)=(x-B(x))^2$ for all $x\in A$.
More concretely, any commutative ring $R$ with identity becomes an involutive
Alexander singquandle by choosing elements $t,B\in R$ such that $t^2=1$,
$B(1+t)=0$ and $t-(1-B)^2=0$ and setting
\[x\ast y=tx+(1-t)y, \quad R_1(x,y)=(1-t-B)x+(t+B)y 
\quad \mathrm{and}\quad R_2(x,y)=(1-B)x + By.\]

For example, in $\mathbb{Z}_5$ with $t=4$ and $B=3$ we have $t^2=4^2=16=1$,
$B(1+t)=3(1+4)=15=0$ and $t-(1-B)^2=4-(1-3)^2=4-4=0$, so $X=\mathbb{Z}_5$ is
an involutive Alexander singquandle with 
\[x\ast y=4x+2y, \quad R_1(x,y)=4x+2y \quad \mathrm{and}\quad R_2(x,y)=3x+3y.\]
\end{example}

\section{Applications}\label{sec5}

Let $X$ be a $\Lambda$-module with the operations 
\[\begin{array}{rcl}
x\ast y & = & tx+(1-t)y\\ 
R_1(x,y)& = & (1-t-B)x+(t+B)y \quad \mathrm{and} \\
R_2(x,y)& = & (1-B)x + By
\end{array}\]
where $\Lambda=\mathbb{Z}[t,B]/(t^2-1,B(1+t),t-(1-B)^2)$. 

\begin{example}

	\begin{figure}[H]
\tiny{
  \centering
   {\includegraphics[scale=0.6]{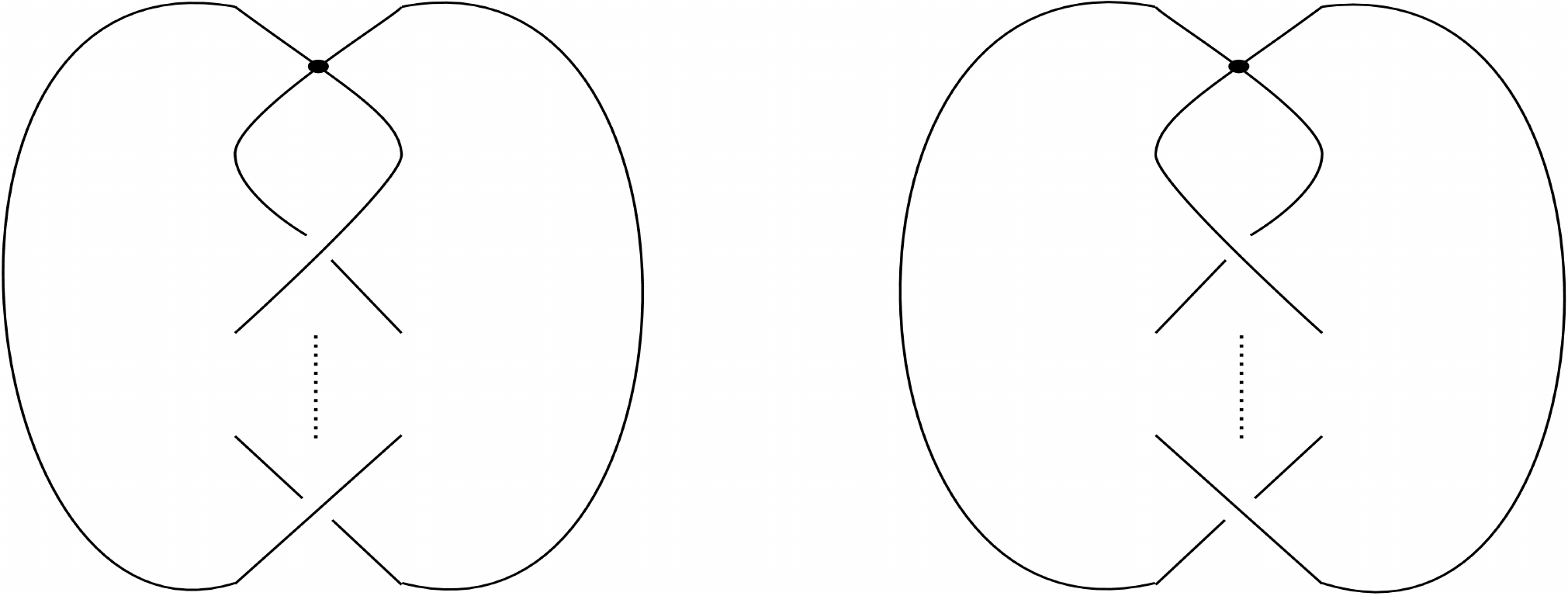}
  }
  \put(-125,130){$x$}
  \put(-20,130){$y$}
    \put(-325,134){$x$}
  \put(-240,134){$y$}
  \vspace{.2in}
     \caption{One singular crossing followed by $(n+1)$ regular crossings}
  \label{singular torus}}
\end{figure}
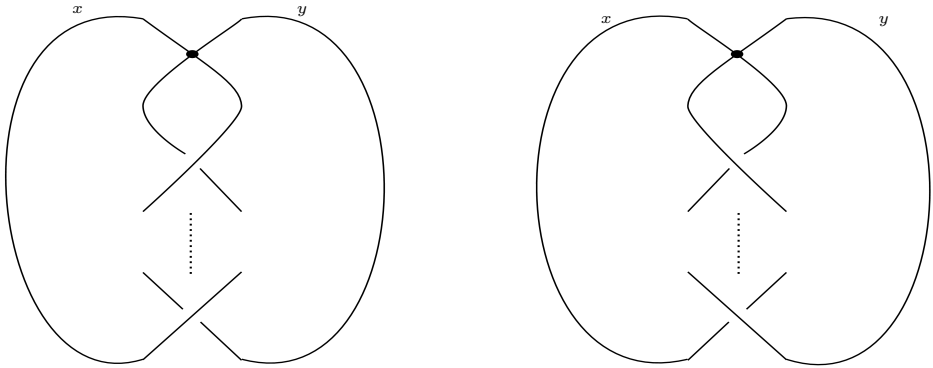	  
We color the top two strands of the knot on the left side of Figure~\ref{singular torus} and assume that the integer $n=2k$ is even.

For the knot on the left then the condition of colorability gives the following system of equations 
\begin{eqnarray*}
x & = &  (-k+kt)R_1(x,y)+ (k+1-kt)R_2(x,y)\\
y & =& [-k +(k+1)t]R_1(x,y) + [k+1-(k+1)t]R_2(x,y).
\end{eqnarray*}
This system of equations simplifies to 
\begin{eqnarray*}
        (1-B)^2(x-y) & = & 0\\
(-kt+B+k)(x-y) & = & 0.
\end{eqnarray*}
thus forcing $x=y$ and consequently in this case the set of colorings is the diagonal inside $X \times X$.

For the knot on the right we also assume $n=2k$. Then the condition of colorability gives the following system of equations
\begin{eqnarray*}
x & = &  (k-kt)R_1(x,y)+ (1-k+kt)R_2(x,y)\\
y & = & [k+1-kt]R_1(x,y) + [-k+kt]R_2(x,y).
\end{eqnarray*}
which simplifies to 
\begin{eqnarray*}
(-k+kt+B)(x-y)& = & 0\\
(-1+k+t-k+B)(x-y)& = & 0.
\end{eqnarray*} implying that  
\begin{eqnarray*}
-k+kt+B & = & 0\\
-1+k+t-k+B & = & 0.
\end{eqnarray*}  
Now by choosing $t=1$, one obtains the single equation $B(x-y)=0$.  Since $B$ doesn't have to be invertible, a right choice of a zero divisor value for $B$ will give that the coloring space contains the diagonal of $X \times X$ as a proper subset and thus the two links will be distinguished by the coloring sets.

\end{example}

\begin{example}
	
				\begin{figure}[H]
\small{
  \centering
   {\includegraphics[scale=0.6]{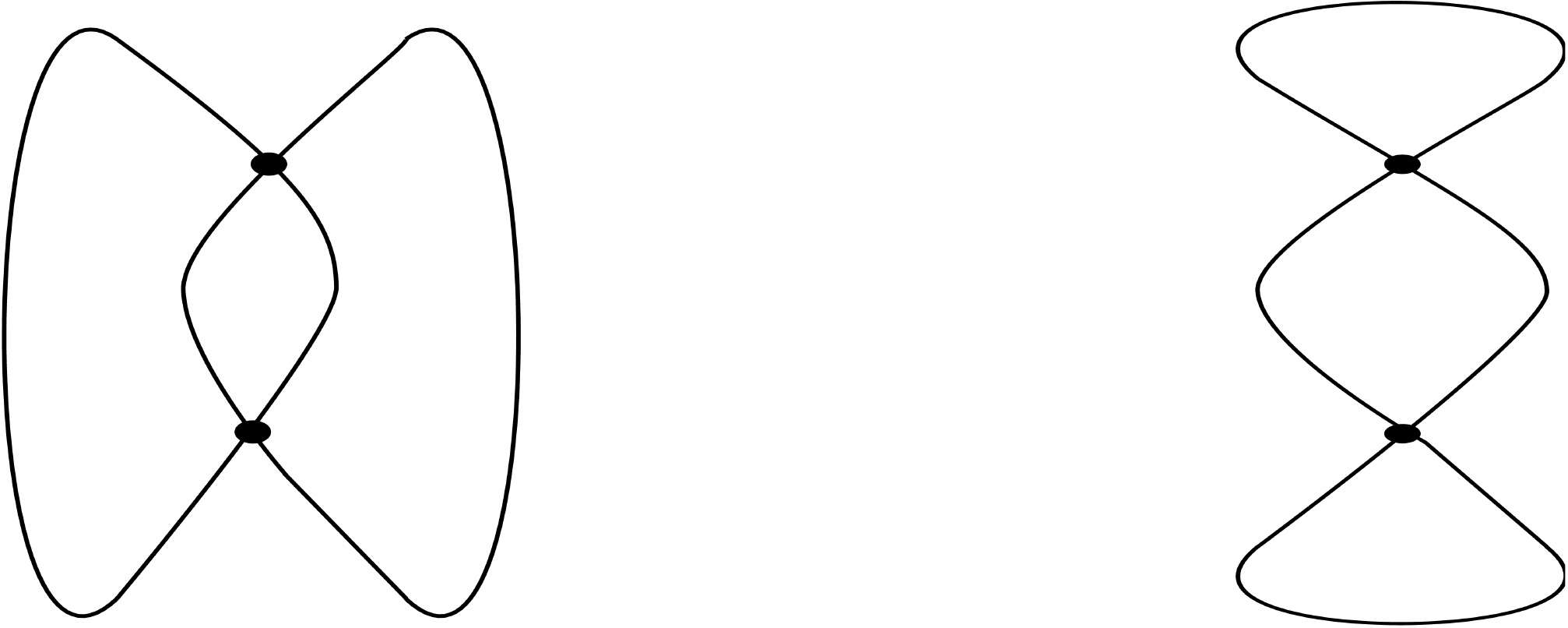}
  }
  \put(-85,80){$x$} 
	\put(-3,80){$x$}
	\put(-30,-10){$x$}
	\put(-30,140){$x$}
	\put(-269,137){$y$} 
	\put(-340,137){$x$}
	\vspace{.2in}
     \caption{Two singular knots each with two singular crossings}
  \label{Hopf}}
\end{figure}
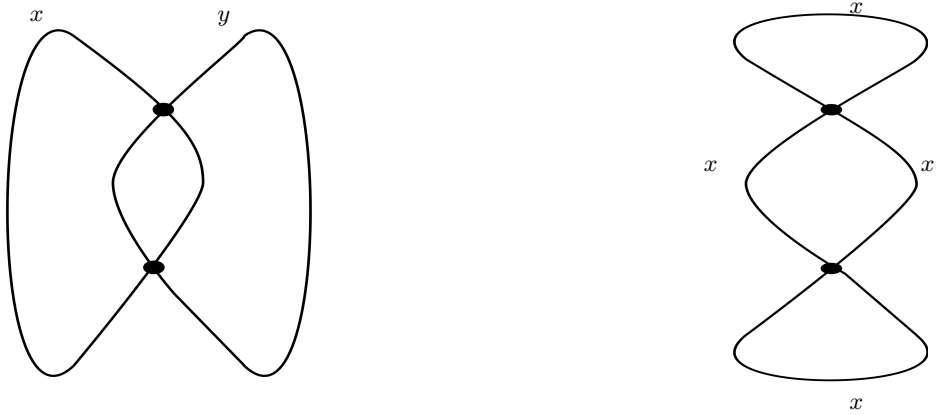

Consider the knot on the left of Figure~\ref{Hopf} and color the two top arcs by elements $x$ and $y$.  By writing the relations at the crossings we get the two equations 
\begin{eqnarray*}
x & =& R_1(R_1(x,y),R_2(x,y))\quad \mathrm{and}\\ 
y & = & R_2(R_1(x,y),R_2(x,y)).
\end{eqnarray*}
We thus obtain after simplification $(t-1+2B)(y-x)=0$. 
Thus the set of all colorings of the knot on the left is \[\{(x,y) \in X \times X\ |\ (t-1+2B)(y-x)=0\}.\] 
It is clear that the knot on the right of Figure \ref{Hopf} colors trivially by the whole quandle. By choosing $\mathbb{Z}_{10}[t,B]$ with $t=-1$ and $B=4$, thus the coloring invariant distinguishes these two singular knots. 
				
			\end{example}

\begin{example}
For our final example, we computed singquandle colorings for certain singular
knots known as \textit{two-bouquet graphs} using the singquandle structure
on the set $X=\{1,2,3,4,5\}$ specified by the following operation tables
where $R_1(x,y)=R_2(x,y)=x\ast' y$:
\[
\begin{array}{r|rrrrr}
\ast & 1 & 2 & 3 & 4 & 5 \\ \hline
1 & 1 & 3 & 5 &  2 & 4 \\
2 & 5 & 2 & 4 & 1 & 3 \\
3 & 4 & 1 & 3 & 5 & 2 \\
4 & 3 & 5 & 2 & 4 & 1 \\
5 & 2 & 4 & 1 & 3 & 5 
\end{array}\quad
\begin{array}{r|rrrrl}
\ast' & 1 & 2 & 3 & 4 & 5 \\ \hline
1 & 1 & 4 & 2 & 5 & 3 \\
2 & 4 & 2 & 5 & 3 & 1 \\
3 & 2 & 5 & 3 & 1 & 4 \\
4 & 5 & 3 & 1 & 4 & 2 \\
5 & 3 & 1 & 4 & 2 & 5.
\end{array}
\]
Then the singular knot on the left has five colorings by $X$
while the one on the right has 25:
\[\includegraphics{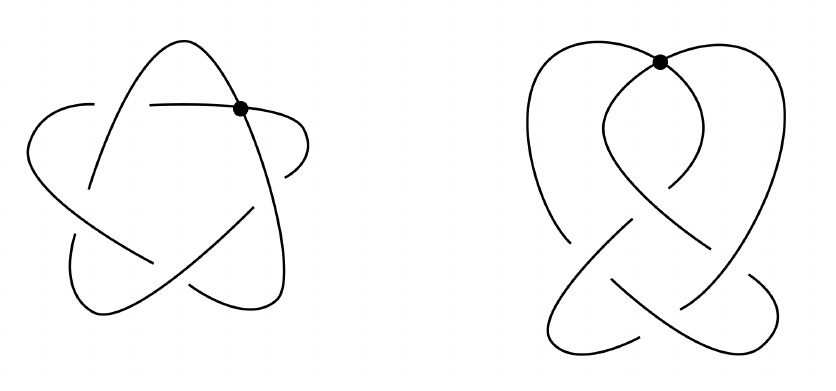}\]
\end{example}

\end{document}